\documentclass[12pt,reqno]{amsart}

\usepackage{multirow}
\usepackage{hhline}
\usepackage{float}

\usepackage{mathtools}
\usepackage{times}
\usepackage[T1]{fontenc}
\usepackage{mathrsfs}
\usepackage{latexsym}
\usepackage[dvips]{graphics}
\usepackage[titletoc, title]{appendix}
\usepackage{amsmath,amsfonts,amsthm,amssymb,amscd}
\usepackage[dvipsnames]{xcolor}
\usepackage{hyperref}
\usepackage{amsmath}
\usepackage[utf8]{inputenc}

\usepackage{color}
\usepackage{breakurl}

\usepackage{comment}
\newcommand{\bburl}[1]{\textcolor{blue}{\url{#1}}}

\usepackage{caption}

\newtheorem{thm}{Theorem}[section]

\newtheorem{cor}[thm]{Corollary}

\newtheorem{lem}[thm]{Lemma}
\newtheorem{prop}[thm]{Proposition}

\newtheorem{rek}[thm]{Remark}

\usepackage[utf8]{inputenc}

\DeclareFixedFont{\ttb}{T1}{txtt}{bx}{n}{12} 
\DeclareFixedFont{\ttm}{T1}{txtt}{m}{n}{12}  

\usepackage{color}
\definecolor{deepblue}{rgb}{0,0,0.5}
\definecolor{deepred}{rgb}{0.6,0,0}
\definecolor{deepgreen}{rgb}{0,0.5,0}

\usepackage{listings}

\newcommand\pythonstyle{\lstset{
language=Python,
basicstyle=\ttm,
morekeywords={self},              
keywordstyle=\ttb\color{deepblue},
emph={MyClass,__init__},          
emphstyle=\ttb\color{deepred},    
stringstyle=\color{deepgreen},
frame=tb,                         
showstringspaces=false
}}

\lstnewenvironment{python}[1][]
{
\pythonstyle
\lstset{#1}
}
{}

\newcommand\pythoninline[1]{{\pythonstyle\lstinline!#1!}}

\definecolor{ao}{rgb}{0.0, 0.5, 0.0}

\newcommand{\lowersub}[2][0.7ex]{%
  _{\raisebox{-#1}{$\scriptstyle #2$}}%
}

\numberwithin{equation}{section}

\DeclareFontFamily{U}{mathx}{}
\DeclareFontShape{U}{mathx}{m}{n}{<-> mathx10}{}
\DeclareSymbolFont{mathx}{U}{mathx}{m}{n}
\DeclareMathAccent{\widehat}{0}{mathx}{"70}
\DeclareMathAccent{\widecheck}{0}{mathx}{"71}

\begin{document}

\title{Integers Having $F_{2k}$ in Both Zeckendorf And Chung-Graham Decompositions}

\author[L. Bustos]{Lucas Bustos}
\email{\textcolor{blue}{\href{mailto: lbustos1007@tamu.edu}{lbustos1007@tamu.edu}}}
\address{Department of Mathematics\\ Texas A\&M University, College Station, TX 77843, USA}

\author[H. V. Chu]{H\`ung Vi\d{\^e}t Chu}
\email{\textcolor{blue}{\href{mailto:hchu@wlu.edu}{hchu@wlu.edu}}}
\address{Department of Mathematics\\ Washington and Lee University, Lexington, VA 24450, USA}

\author[M. Kim]{Minchae Kim}
\email{\textcolor{blue}{\href{mailto:alsco1103@tamu.edu}{alsco1103@tamu.edu}, 
\href{mailto:claireminchaekim@gmail.com}{claireminchaekim@gmail.com}}}
\address{Department of Mathematics\\ Texas A\&M University, College Station, TX 77843, USA}

\author[U. Lee]{Uihyeon Lee}
\email{\textcolor{blue}{\href{mailto:uihyeonlee@tamu.edu}{uihyeonlee@tamu.edu}, \href{mailto:uihyeon8711@gmail.com}{uihyeon8711@gmail.com}}}
\address{Department of Mathematics\\ Texas A\&M University, College Station, TX 77843, USA}

\author[S. Shankar]{Shreya Shankar}
\email{\textcolor{blue}{\href{mailto:shreyashankar_11@tamu.edu}{shreyashankar\_11@tamu.edu}}}
\address{Department of Mathematics\\ Texas A\&M University, College Station, TX 77843, USA}

\author[G. Tresch]{Garrett Tresch}
\email{\textcolor{blue}{\href{mailto:treschgd@tamu.edu}{treschgd@tamu.edu}}}
\address{Department of Mathematics\\ Texas A\&M University, College Station, TX 77843, USA}

\thanks{This work was partially supported by the College of Arts \& Sciences at Texas A\&M University. The first, third, fourth, and fifth named authors are undergraduate students at Texas A\&M University at the time of the paper, working under the guidance of the second and sixth named authors.}



\maketitle
 
\begin{abstract}
Zeckendorf's theorem states that every positive integer can be uniquely decomposed into nonadjacent Fibonacci numbers. On the other hand, Chung and Graham proved that every positive integer can be uniquely written as a sum of even-indexed Fibonacci numbers with coefficients $0,1$, or $2$ such that between two coefficients $2$, there is a coefficient $0$. We discover a correspondence between a lexicographically ordered sublist of Zeckendorf decompositions and letters in the golden string $\mathcal{S}$. Likewise, we identify a dual correspondence for Chung-Graham decompositions. We then use these correspondences to give the set of all positive integers having $F_{2k}$ in both of their Zeckendorf and Chung-Graham decompositions.   
\end{abstract}

\tableofcontents


\section{Introduction}
For $n\in \mathbb{N}\cup \{0\}$, let $F_n$ be the $n$th Fibonacci number, i.e., $F_{n+2} = F_{n+1} + F_n$ for all $n\ge 0$ with $F_0 = 0$ and $F_1 = 1$. Zeckendorf's theorem \cite{Ze} states that every positive integer can be written uniquely as a sum of nonadjacent Fibonacci numbers in $(F_n)_{n=2}^\infty$. This sum is often called the \textit{Zeckendorf decomposition} of integers, which can be found by the greedy algorithm. For example, the largest Fibonacci number in the Zeckendorf decomposition of $28$ is $F_{8} = 21$ because $F_8$ is the largest Fibonacci number not greater than $28$. The next term in the decomposition is the largest Fibonacci number at most $28-21 = 7$, which is $F_5 = 5$. The remainder after taking out $F_5$ is $7-5 = 2 = F_3$; hence,
$$28 \ =\ 21 + 5 + 2 \ =\ F_8 + F_5 + F_3.$$
There has been extensive research on the structure of the Zeckendorf decomposition and its generalizations. For a sample of notable work, see \cite{Ba, CHS, CFHMN, Cha1, Cha11, Day1, Day2, De1, De2, DDKMMV, DFFHMPP, DS,  G10, HW, Ho, MMMS, MMMMS, Sh}.

On the other hand, Chung and Graham \cite{CG} proved that every positive integer can be uniquely written as a sum of even-indexed Fibonacci numbers using coefficients $0, 1$, or $2$ such that between two coefficients $2$, there exists a coefficient $0$. We call the sum the \textit{Chung-Graham decomposition} of $n$. 
\begin{thm}\cite[Lemma 1]{CG}
Every positive integer $n$ can be uniquely represented as a sum
$$n\ =\ \sum_{i = 1}^\infty c_i F_{2i}, \mbox{ where } c_i\in \{0, 1, 2\},$$
so that if $c_{i_1} = c_{i_2} = 2$, then there exists $j$ between $i_1$ and $i_2$ such that $c_j = 0$.
\end{thm}
The proof of \cite[Lemma 1]{CG} showed that the Chung-Graham decomposition is also determined by the greedy algorithm. At each step, we find the largest $j$ such that $F_{2j}$ is not greater than the remainder. If $2F_{2j}$ is also less than the remainder, the next term is $2F_{2j}$; otherwise, the next term is $F_{2j}$. For example, 
$$141 \ =\ 2F_{10} + F_{8} + F_6 + 2F_2.$$

Recent work generalized the Chung-Graham decomposition to  use equally-spaced Fibonacci numbers \cite{Cha2}, studied the conversion between Zeckendorf and Chung-Graham decompositions \cite{Bu}, and characterized integers without a fixed term in their Chung-Graham decompositions \cite{CKV}, thus continuing the work by Kimberling \cite{Kim} who characterized integers without $F_2$ in their Zeckendorf decompositions and by Griffiths \cite{G14} who characterized integers having $F_k$ in their Zeckendorf decompositions. These results inspired us to 
discover a correspondence between a lexicographically ordered sublist of Zeckendorf decompositions and letters in the golden string $\mathcal{S}$. Likewise, we identify a dual correspondence for Chung-Graham decompositions. We then use these correspondences to characterize integers that have $F_{2k}$ in both Zeckendorf and Chung-Graham decompositions.

To formally state our results, let us recall the golden string $\mathcal{S}$. We start with $S_1 = B$, $S_2 = BA$, and $S_n = S_{n-1}:S_{n-2}$ for $n\ge 3$, where $S_{n-1}:S_{n-2}$ is the concatenation of the two finite strings. For example, 
\begin{align*}
    S_3 &\ =\ S_2: S_1\ =\ BAB,\\
    S_4 &\ =\ S_3:S_2\ =\ BABBA, \mbox{ and}\\
    S_5 &\ =\ S_4:S_3\ =\ BABBABAB.
\end{align*}
It follows that $|S_n| = F_{n+1}$ for all $n\ge 1$. The golden string $\mathcal{S}$ is the infinite string whose first $F_{n+1}$ letters is the same as the finite string $S_n$. In other words, $\mathcal{S} = \lim_{n\rightarrow \infty} S_n$ in the sense that for every $i$, the $i$th letter of $\mathcal{S}$, denoted by $\mathcal{S}(i)$, is the same as the $i$th letter of $S_n$ for sufficiently large $n$. The first few letters of $\mathcal{S}$ are
$$BABBABABBABBABABBABAB\ldots\,.$$

Throughout the paper, we fix $k\ge 1$. Let $A_{2k}$ be the set of all positive integers whose Zeckendorf decomposition has $F_{2k}$ as the smallest summand. Let $p(j)$ be the $j$th smallest integer in $A_{2k}$. Similarly, let $B_{2k}$ be the set of all positive integers whose Chung-Graham decomposition has $F_{2k}$ or $2F_{2k}$ as the smallest summand. Let $q(j)$ be the $j$th smallest integer in $B_{2k}$. For a positive integer $n$, we use $Z(n)$ and $CG(n)$ to denote the set of all Fibonacci numbers in the Zeckendorf decomposition and the Chung-Graham decomposition of $n$, respectively.

\begin{table}[H]
\centering
\begin{tabular}{ c|c c c c c c c |c |c }
$A_{2k}$& & & &$Z(p(j))$ & & & & $F_{2k}\in CG(p(j))$ ?& $\mathcal{S}$\\
\hline
$p(1)$&$F_{2k}$ & & & & & & &  Y & \\ 
$p(2)$&$F_{2k}$ & & $F_{2k+2}$ & & & & & Y & B\\ 
$p(3)$&$F_{2k}$ & & & $F_{2k+3}$ & & & & N & A\\ 
$p(4)$&$F_{2k}$ & & & & $F_{2k+4}$ & & & Y & B\\
$p(5)$&$F_{2k}$ & &$F_{2k+2}$ & & $F_{2k+4}$ & & & Y & B\\
$p(6)$&$F_{2k}$ & && & & $F_{2k+5}$ & & N & A\\ 
$p(7)$&$F_{2k}$ && $F_{2k+2}$ & & & $F_{2k+5}$ & & Y & B\\
$p(8)$&$F_{2k}$ & && $F_{2k+3}$ & &  $F_{2k+5}$ & & N &A\\
$p(9)$&$F_{2k}$ & && & & & $F_{2k+6}$ & Y & B\\
$p(10)$&$F_{2k}$ && $F_{2k+2}$ & & & & $F_{2k+6}$ & Y & B\\
$p(11)$&$F_{2k}$ & && $F_{2k+3}$ & & &  $F_{2k+6}$ & N & A\\
$p(12)$&$F_{2k}$ & & & & $F_{2k+4}$ & &  $F_{2k+6}$ & Y & B\\
$p(13)$&$F_{2k}$ & &$F_{2k+2}$& & $F_{2k+4}$ & &  $F_{2k+6}$ & Y & B\\
  \vdots& & & & & & & &
\end{tabular}
\caption{Integers $p(j)$ in $A_{2k}$. The first column specifies $p(j)$, while the second column gives the Zeckendorf decomposition of $p(j)$. The third column tells whether $F_{2k}\in CG(p(j))$, where Y and N stand for Yes and No, respectively. The golden string $\mathcal{S}$ is placed in the fourth column to be compared with the third column.}
\label{ZtablewithCG}
\end{table}

\begin{table}[H]
\centering
\begin{tabular}{ c| c c c c c c |c |c }
$B_{2k}$& & & $CG(q(j))$& & & & $F_{2k}\in Z(q(j))$? & $\mathcal{S}$\\
\hline
$q(1)$&$F_{2k}$ & & & &  & & Y &\\ 
$q(2)$&$2F_{2k}$ & & & &  & & N & B\\  
$q(3)$&$F_{2k}$ & & $F_{2k+2}$ & & & & Y & A\\  
$q(4)$&$2F_{2k}$ & &  $F_{2k+2}$ & & & & N & B\\  
$q(5)$&$F_{2k}$ & &  $2F_{2k+2}$ & & & & N & B\\  
$q(6)$&$F_{2k}$ &  &  & & $F_{2k+4}$ & & Y & A\\ 
$q(7)$&$2F_{2k}$ &  & & & $F_{2k+4}$ & & N & B\\  
$q(8)$&$F_{2k}$ & & $F_{2k+2}$ & & $F_{2k+4}$ & & Y & A \\  
$q(9)$&$2F_{2k}$ & & $F_{2k+2}$ & & $F_{2k+4}$ & & N & B\\  
$q(10)$&$F_{2k}$ & & $2F_{2k+2}$ & & $F_{2k+4}$ & & N & B\\  
$q(11)$&$F_{2k}$ &  & & &  $2F_{2k+4}$ & & Y & A\\ 
$q(12)$&$2F_{2k}$ &  & & & $2F_{2k+4}$ & & N & B\\  
$q(13)$&$F_{2k}$ & &  $F_{2k+2}$ & & $2F_{2k+4}$ & & N & B\\  
 \vdots& & & & & & & &
\end{tabular}
\caption{Integers $q(j)$ in $B_{2k}$. The first column specifies $q(j)$, while the second column gives the Chung-Graham decomposition of $q(j)$. The third column tells whether $F_{2k}\in Z(q(j))$, where Y and N stand for Yes and No, respectively. The golden string $\mathcal{S}$ is placed in the fourth column to be compared with the third column.}
\label{CGtablewithZ}
\end{table}

Tables \ref{ZtablewithCG} and \ref{CGtablewithZ} list the values of $A_{2k}$ and $B_{2k}$, and they suggest the following duality of correspondences between the decompositions and the golden string $\mathcal{S}$.

\begin{thm}\label{thmZtoCG}
For $j\ge 2$, the integer $p(j)$ in $A_{2k}$ has $F_{2k}\in CG(p(j))$ if and only if $\mathcal{S}(j-1) = B$. Moreover, $\min CG(p(j))\ge F_{2k}$ for all $p(j)\in A_{2k}$. 
\end{thm}

\begin{thm}\label{dual}
For $j\ge 2$, the integer $q(j)$ in $B_{2k}$ has $F_{2k}\in Z(q(j))$ if and only if $\mathcal{S}(j-1) = A$. Moreover, $\min Z(q(j))\ge F_{2k-2}$ for all $q(j)\in B_{2k}$. 
\end{thm}

Using these correspondences, we obtain the complete description of all integers $n$ with $F_{2k}\in Z(n)\cap CG(n)$.

\begin{thm}\label{mainthm}
For $k\ge 1$, the set of positive integers that have $F_{2k}$ in both Zeckendorf and Chung-Graham decompositions is 
$$\{nF_{2k}+\lfloor(n-1)\phi\rfloor F_{2k+1} + j\ :\ n\in \mathbb{N},\, 0\leq j\leq F_{2k-1}-1\}$$
where $\phi = (1+\sqrt{5})/2$, the golden ratio. 
\end{thm}

Our paper is structured as follows. Section \ref{conversion} demonstrates the conversions between the two decompositions. Section \ref{duality} proves the duality of correspondences, namely Theorems \ref{thmZtoCG} and \ref{dual}. Finally, Section \ref{proofmt} is devoted to prove Theorem \ref{mainthm}. The proofs of the conversion statements in Section \ref{conversion} are found in Appendix \ref{apencon}.

\section{Conversion between the two decompositions}\label{conversion}

We discuss conversions between the two decompositions. Each conversion result can be proved by induction on the number of terms. In this section, however, we emphasize the mechanics of the conversions, while their proofs are provided in Appendix \ref{apencon} for interested readers.

Let us use the subscript tags $\mbox{ }\lowersub{\alpha Z}$ and $\mbox{ }\lowersub{\alpha CG}$ to indicate that a decomposition is a Zeckendorf or Chung-Graham decomposition, respectively; for example, 
$$F_{2} + F_4 + 2F_8 + F_{10} + 2F_{14}\mbox{ }\lowersub{\alpha CG}\ =\ F_4 + F_7 + F_{10} + F_{12} + F_{15}\mbox{ }\lowersub{\alpha Z}.$$

\begin{lem}\label{l1}
    For $N\in \mathbb{N}$, let $\sum_{m=1}^Na_mF_{2k+2m+1}$ be a Zeckendorf decomposition with $a_N = 1$. Then there exist coefficients $b$ and $(b_m)_{m=2}^{N}$ with $b, b_N\in \{1,2\}$ such that
    \begin{equation}\label{e1}F_{2k} + \sum_{m=1}^N a_mF_{2k+2m+1}\mbox{ }\lowersub{\alpha Z}\ =\ bF_{2k+2} + \sum_{m=2}^N b_mF_{2k+2m}\mbox{ }\lowersub{\alpha CG}.\end{equation}
\end{lem}

\begin{cor}[Conversion to Chung-Graham decompositions]\label{toCG} Let $r\ge 1$ and $N\ge 2r$. 
Given each Zeckendorf decomposition below, there exist coefficients $b\in \{1,2\}$ and $(b_m)_{2\le m\le 2\lfloor N/2\rfloor}$ such that
\begin{equation}\label{e3}
    F_{2k} + F_{2k+2r+1} + \sum_{m=2r+3}^N a_m F_{2k+m}\mbox{ }\lowersub{\alpha Z}\ =\ bF_{2k+2} + \sum_{m=4}^{2\lfloor N/2\rfloor} b_m F_{2k+m}\mbox{ }\lowersub{\alpha CG}
\end{equation}
and 
\begin{equation}\label{e4}
    F_{2k} + F_{2k+2r} + \sum_{m=2r+2}^N a_m F_{2k+m}\mbox{ }\lowersub{\alpha Z}\ =\ F_{2k} + \sum_{m=2}^{2\lfloor N/2\rfloor} b_m F_{2k+m}\mbox{ }\lowersub{\alpha CG}.
\end{equation}
\end{cor}

\begin{lem}[Conversion to Zeckendorf decompositions]\label{toZeck0}
    Let $M\ge 0$ be even and $b_M\in \{1,2\}$. Given each Chung-Graham decomposition below, there exist coefficients $(a_m)_{m=1}^{M+1}$ such that
    \begin{equation}\label{e6}
    2F_{2k} + \sum_{m=2}^M b_m F_{2k+m}\mbox{ }\lowersub{\alpha CG}\ =\ F_{2k-2} + \sum_{m=1}^{M+1} a_m F_{2k+m} \mbox{ }\lowersub{\alpha Z}
    \end{equation}
    and     
    \begin{equation}\label{e7}
    F_{2k} + \sum_{m=2}^M b_m F_{2k+m}\mbox{ }\lowersub{\alpha CG}\ =\ \begin{cases}F_{2k-2} + \sum_{m=1}^{M+1} a_m F_{2k+m}\mbox{ }\lowersub{\alpha Z},\\ F_{2k} + \sum_{m=1}^{M+1} a_m F_{2k+m}\mbox{ }\lowersub{\alpha Z}.\end{cases}
    \end{equation}
    In the case $k = 1$, we remove $F_{2k-2} = 0$ from the Zeckendorf decomposition.
\end{lem}

\section{The golden string and duality}\label{duality}
In this section, we establish the correspondence between the golden string $\mathcal{S}$ and the integers $p(j)\in A_{2k}$ with $\min CG(p(j)) = F_{2k}$. An analog of the correspondence holds between $\mathcal{S}$ and the integers $q(j)\in B_{2k}$ with $\min Z(q(j)) = F_{2k}$.

\begin{lem} \cite[cf.\ Lemma 3.2]{G11}\label{prelem}
Let $n = F_c + \sum_{m=c+2}^N a_m F_m\mbox{ }\lowersub{\alpha Z}$, i.e., $\min Z(n) = c$. Then $\mathcal{S}(n) = B$ if and only if $c$ is even.    
\end{lem}

For each $j\ge 2$, write $p(j) = F_{2k} + \sum_{m=2}^\infty a_m F_{2k+m}\mbox{ }\lowersub{\alpha Z}$ and define the function $p^*: \mathbb{N}_{\ge 2}\rightarrow \mathbb{N}$ as $p^*(j) := \sum_{m=2}^\infty a_m F_m$.

\begin{prop}\label{propp*}
For $j\ge 2$, we have $p^*(j) = j-1$.    
\end{prop}

\begin{proof}
Note that $p^*$ is surjective because sums of nonadjacent Fibonacci numbers in $(F_n)_{n\ge 2}$ give all natural numbers.
We prove that $p^*$ is strictly increasing. 
Let $j < j'$ with 
$$p(j)\ =\ F_{2k} + \sum_{m=2}^\infty a_m F_{2k+m}\mbox{ }\lowersub{\alpha Z}\mbox{ and }p(j')\ =\ F_{2k} + \sum_{m=2}^{\infty} a'_m F_{2k+m}\mbox{ }\lowersub{\alpha Z}.$$
Let $r$ be the largest index such that $a'_r\neq a_r$. The lexicographical order of Zeckendorf decompositions implies that $a'_r = 1$ and $a_r = 0$, and thus, 
$$p^*(j)\ =\ \sum_{m=2}^\infty a_m F_m \ <\ \sum_{m=2}^\infty a'_m F_{m}\ =\ p^*(j').$$
It follows that $p^*(j) = j-1$.
\end{proof}

\begin{proof}[Proof of Theorem \ref{thmZtoCG}]
  The second statement follows from Corollary \ref{toCG}. We prove the first statement. By Corollary \ref{toCG}, $F_{2k}\in CG(p(j))$ if and only if 
  $$p(j)\ =\ F_{2k} + F_{2k+2r} + \sum_{m=2r+2}^\infty a_m F_{2k+m}\mbox{ }\lowersub{\alpha Z},$$
  which gives
  $$p^*(j)\ =\ F_{2r} + \sum_{m=2r+2}^\infty a_m F_m.$$
  By Proposition \ref{propp*}, 
  $$j - 1 \ =\ F_{2r} + \sum_{m=2r+2}^\infty a_m F_m.$$
  By Lemma \ref{prelem},
  $$\mathcal{S}(j-1) \ =\ B.$$
\end{proof}

To prove Theorem \ref{dual}, we need the analog of Proposition \ref{propp*} for $q(j)\in B_{2k}$, which is obtained with the help of the following result. 

\begin{cor}\label{corq} For each $j\ge 2$, there exists $r\ge 1$ such that
$$q(j)\ =\ \begin{cases}F_{2k} + F_{2k+2r} + \sum_{m=2r+2}^\infty a_m F_{2k+m}\mbox{ }\lowersub{\alpha Z},\\ F_{2k-2} + F_{2k+2r-1} + \sum_{m=2r+1}^\infty a_m F_{2k+m}\mbox{ }\lowersub{\alpha Z}.\end{cases}$$   
\end{cor}

\begin{proof}
    By Lemma \ref{toZeck0}, we know that either 
    $$q(j) \ =\ F_{2k} + \sum_{m=2}^{\infty} a_m F_{2k+m}\mbox{ }\lowersub{\alpha Z}\quad\mbox{ or }\quad q(j) \ =\ F_{2k-2} + \sum_{m=1}^{\infty} a_m F_{2k+m}\mbox{ }\lowersub{\alpha Z}.$$

    Case 1: $q(j) = F_{2k} + \sum_{m=2}^{\infty} a_m F_{2k+m}\mbox{ }\lowersub{\alpha Z}$. Since $j\ge 2$, we have $q(j) > F_{2k}$, so $\sum_{m=2}^{\infty} a_m F_{2k+m} > 0$. Let $m^* = \min \{m\ge 2: a_m\neq 0\}$. By Corollary \ref{toCG}, $m^*$ is even, namely $m^* = 2r$ for some $r\ge 1$. In this case,
    $$q(j)\ =\ F_{2k} + F_{2k+2r} + \sum_{m=2r+2}^\infty a_m F_{2k+m}\mbox{ }\lowersub{\alpha Z}.$$

    Case 2: $q(j) = F_{2k-2} + \sum_{m=1}^{\infty} a_m F_{2k+m}\mbox{ }\lowersub{\alpha Z}$. Let $m^* = \min \{m\ge 1:a_m\neq 0\}$. Suppose, for a contradiction, that $m^* = 2r$ for some $r\ge 1$. Then 
    $$q(j) \ =\ F_{2k-2} + F_{2k+2r} +  \sum_{m=2r+2}^{\infty} a_m F_{2k+m}\mbox{ }\lowersub{\alpha Z}.$$
    \begin{enumerate}
        \item[] Case 2.1: If $k\ge 2$, then according to \eqref{e4}, we have $\min CG(q(j)) = F_{2k-2}$, which contradicts that $q(j)\in B_{2k}$.
        \item[] Case 2.2: If $k = 1$, then $\min Z(q(j)) = F_{2k+2r}$. By Corollary \ref{toCG}, $\min CG(q(j))$ is either $F_{2k+2r}$ or $F_{2k+2r+2}$, contradicting $q(j)\in B_{2k}$.
    \end{enumerate}
    Therefore, $m^* = 2r-1$ for some $r\ge 1$, and we have
    $$q(j)\ =\ F_{2k-2} + F_{2k+2r-1} + \sum_{m=2r+1}^{\infty} a_m F_{2k+m}\mbox{ }\lowersub{\alpha Z}.$$
\end{proof}

It follows from Corollary \ref{corq} that each $q(j)\in B_{2k}$ can be written as 
\begin{equation}\label{re22}q(j)\ =\  c_1 F_{2k-2} + c_2F_{2k} + \sum_{m=1}^\infty a_m F_{2k+m}\mbox{ }\lowersub{\alpha Z},\end{equation}
where $c_1, c_2\in \{0,1\}$ and $c_1 + c_2 = 1$.
Define the function $q^*: \mathbb{N}_{\ge 2}\rightarrow \mathbb{N}$ as 
$$q^*(j) \ :=\ \sum_{m=1}^N a_m F_{m+1}.$$

We have the analog of Proposition \ref{propp*}.
\begin{prop}\label{propq*}
For $j\ge 2$, $q^*(j) = j-1$.    
\end{prop}

\begin{proof}
First, we prove that $q^*$ is surjective. Let $n\in \mathbb{N}$ with $n = \sum_{m=1}^\infty a_m F_{m+1}\mbox{ }\lowersub{\alpha Z}$. Let $m^* = \min \{m\ge 1: a_m\neq 0\}$. 

Case 1: $m^*$ is odd. By \eqref{e3}, 
$$F_{2k-2} + F_{2k+m^*} + \sum_{m=m^*+2}^\infty a_m F_{2k+m}\mbox{ }\lowersub{\alpha Z} \ \in\ B_{2k}.$$

Case 2: $m^*$ is even. By \eqref{e4}, 
$$F_{2k} + F_{2k+m^*} + \sum_{m = m^*+2}^\infty a_m F_{2k+m}\mbox{ }\lowersub{\alpha Z}\ \in\ B_{2k}.$$

We see from both cases that exists $j\ge 2$ with 
$$q^*(j) \ =\ F_{m^*+1} + \sum_{m=m^*+2}^\infty a_m F_{m+1}\ =\ n.$$

Second, we prove that $q^*$ is strictly increasing. Let $2\le j < j'$. According to \eqref{re22}, 
we write
$$q(j)\ =\ c_1F_{2k-2} + c_2F_{2k} + \sum_{m=1}^\infty a_m F_{2k+m}\mbox{ }\lowersub{\alpha Z}$$
and
$$q(j')\ =\ c'_1F_{2k-2} + c'_2F_{2k} + \sum_{m=1}^\infty a'_m F_{2k+m} \mbox{ }\lowersub{\alpha Z}.$$

If $\sum_{m=1}^\infty a_m F_{2k+m} =  \sum_{m=1}^\infty a'_m F_{2k+m}$, then $q(j) < q(j')$ implies that $(c_1, c_2) = (1,0)$ and $(c'_1, c'_2) = (0,1)$. By Corollary \ref{corq}, 
$\min \{m\ge 1: a_m = 1\}$ is odd, while $\min \{m\ge 1: a'_m  = 1\}$ is even. This contradicts that $\sum_{m=1}^\infty a_m F_{2k+m} =  \sum_{m=1}^\infty a'_m F_{2k+m}$.

If $\sum_{m=1}^\infty a_m F_{2k+m} \neq  \sum_{m=1}^\infty a'_m F_{2k+m}$, then $q(j) < q(j')$ implies that if $M$ is the largest integer such that $a_M\neq a'_M$, then $a'_M = 1$ and $a_M = 0$.
It follows that $q^*(j') > q^*(j)$.
\end{proof}

\begin{proof}[Proof of Theorem \ref{dual}]
    The second statement follows from Corollary \ref{corq}. We prove the first statement. By Corollary \ref{corq}, $F_{2k}\in Z(q(j))$ if and only if 
    $$q(j)\ =\ F_{2k} + F_{2k+2r} + \sum_{m=2r+2}^\infty a_m F_{2k+m}\mbox{ }\lowersub{\alpha Z},$$
    which, by Proposition \ref{propq*}, gives
    $$j-1 \ =\ q^*(j)\ =\ F_{2r+1} + \sum_{m=2r+2}^\infty a_mF_{m+1}.$$
    By Lemma \ref{prelem}, $\mathcal{S}(j-1) = A$.
\end{proof}

\begin{rek}\normalfont
As we see from the proof of Theorems \ref{thmZtoCG} and \ref{dual}, the index $m+1$ in the definition of $q^*$ (instead of $m$ as in the definition of $p^*$) explains why integers in $B_{2k}$ correspond to $A$'s in $\mathcal{S}$, while integers in $A_{2k}$ correspond to $B$'s in $\mathcal{S}$. 
\end{rek}

\section{Integers that have $F_{2k}$ in both decompositions}\label{proofmt}

Let $C_{2k} = A_{2k} \cap B_{2k}$. Define $r(j)$ to be the $j$th smallest number in $C_{2k}$. Let $\beta(j)$ denote the position of the $j$th appearance of the letter $B$ in $\mathcal{S}$.
We record several useful properties of $\mathcal{S}$ to be used in due course.

\begin{lem} \label{propertyS} Let $n\ge 1$.
\begin{enumerate}
    \item\label{ps8} If $\mathcal{S}(n) = A$, then $\mathcal{S}(n+1) = B$. 
    \item\label{ps2} The number of $B$'s in the first $n$ letters of $\mathcal{S}$ is $\left\lfloor \frac{n+1}{\phi}\right\rfloor$ (\cite[Lemma 3.3]{G11}).
    \item\label{ps3} The $n$th $B$ of $\mathcal{S}$ occurs at position $\lfloor n\phi\rfloor$, i.e., $\beta(n) = \lfloor n\phi\rfloor$ (\cite[Lemma 3.4]{G11}).
\end{enumerate}
\end{lem}

\begin{proof}
    Item \eqref{ps8} is easily proved by induction, while interested readers may refer to \cite{G11} for the proof of items \eqref{ps2} and \eqref{ps3}.
\end{proof}

Next, we have the correspondence between terms in $A_{2k}$ and letters of $\mathcal{S}$ and establish its analog for $C_{2k}$.

\begin{lem}\cite[Lemma 3.2]{G14}\label{lemma3.2G}
    For $j\ge 1$,
    $$p(j+1)-p(j)\ =\ \begin{cases}F_{2k+1}, \mbox{ if }A\mbox{ is the }j\mbox{th character of }\mathcal{S};\\ F_{2k+2}, \mbox{ if }B\mbox{ is the }j\mbox{th character of }\mathcal{S}.\end{cases}$$
\end{lem}

\begin{lem} \label{condiff} For $j\ge 2$,
$$r(j+1) - r(j)\ =\ \begin{cases}F_{2k+2},\mbox{ if }B\mbox{ is the }(\beta(j)-1)\mbox{th}\mbox{ letter of }\mathcal{S};\\ F_{2k+3}, \mbox{ if }A\mbox{ is the }(\beta(j)-1)\mbox{th}\mbox{ letter of }\mathcal{S}.\end{cases}$$
\end{lem}

\begin{proof}
    By Theorem \ref{thmZtoCG}, we have
$$r(j+1)\ =\ p(\beta(j)+1)\mbox{ for all }j\ge 1.$$
Hence, 
$$r(j+1) - r(j)\ = \ p(\beta(j)+1) - p(\beta(j-1)+1)\mbox{ for all }j\ge 2.$$

Case 1: If $\mathcal{S}(\beta(j)-1) = B$, then 
$\beta(j-1) = \beta(j)-1$. By Lemma \ref{lemma3.2G}, 
$$r(j+1) - r(j)\ =\ p(\beta(j)+1)-p(\beta(j))\ =\ F_{2k+2}.$$

Case 2: If $\mathcal{S}(\beta(j)-1) = A$, then 
Lemma \ref{propertyS}, item \eqref{ps8} gives $\beta(j-1) = \beta(j)-2$. By Lemma \ref{lemma3.2G}, we have 
\begin{align*}
    r(j+1) - r(j)&\ =\ p(\beta(j)+1)-p(\beta(j)-1)\\
    &\ =\ p(\beta(j)+1) - p(\beta(j)) + p(\beta(j)) - p(\beta(j)-1)\\
    &\ =\ F_{2k+2} + F_{2k+1}\ =\ F_{2k+3}. 
\end{align*}
\end{proof}

\begin{proof}[Proof of Theorem \ref{mainthm}]
    For $n\ge 1$, let $a(n)$ denote the number of $A$'s in the first $n$ letters of $\mathcal{S}$. According to Lemma \ref{condiff}, for $j\ge 2$, the difference $r(j+1)-r(j)$ depends on whether the $j$th appearance of $B$ is after an $A$ or a $B$ in $\mathcal{S}$: in the case of $AB$, $r(j+1) - r(j) = F_{2k+3}$, while in the case of $BB$, $r(j+1)-r(j) = F_{2k+2}$. Therefore, the number of times that $r(j+1) - r(j) = F_{2k+3}$ is equal to the number of $A$'s in the first $\beta(n)$ letters of $\mathcal{S}$. By Lemma \ref{propertyS}, item \eqref{ps3}, we have for $n\ge 1$, 
\begin{align*}
    r(n+1)&\ =\ a(\beta(n))F_{2k+3} + (n-a(\beta(n))-1)F_{2k+2} + r(2)\\
    &\ =\ a(\lfloor n\cdot \phi\rfloor)F_{2k+3} + (n-a(\lfloor n\cdot \phi\rfloor)-1)F_{2k+2} + r(2)\\
    &\ =\ a(\lfloor n\cdot \phi\rfloor)F_{2k+1} + (n-1)F_{2k+2} + F_{2k} + F_{2k+2}\\
    &\ =\ F_{2k} + a(\lfloor n\cdot \phi\rfloor)F_{2k+1} + nF_{2k+2}. 
\end{align*}
Hence, 
\begin{equation*}
    C_{2k}\ =\ \{F_{2k}+a\left(\lfloor(n-1)\cdot \phi\rfloor\right)F_{2k+1}+(n-1)F_{2k+2}\ :\ n\in \mathbb{N}\}.
\end{equation*}
It follows from Lemma \ref{propertyS}, item \eqref{ps2} that 
\begin{align*}
a\left(\lfloor(n-1)\cdot \phi\rfloor\right)&\ =\ \lfloor(n-1)\cdot \phi\rfloor - \left\lfloor \frac{\lfloor(n-1)\cdot \phi\rfloor+1}{\phi}\right\rfloor
&\ =\ \lfloor(n-1)\cdot \phi\rfloor - (n-1),
\end{align*}
because
$$n-1\ =\ \bigg\lfloor\dfrac{(n-1)\phi}{\phi}\bigg\rfloor\ \le\ \bigg\lfloor\dfrac{\lfloor(n-1)\phi\rfloor+1}{\phi}\bigg\rfloor\ \le\ \bigg\lfloor\dfrac{(n-1)\phi+1}{\phi}\bigg\rfloor\ =\ n-1.$$
As a result,
$$C_{2k}\ =\ \{nF_{2k}+\lfloor(n-1)\phi\rfloor F_{2k+1}\ :\ n\in \mathbb{N}\}.$$

Using $C_{2k}$, we now find the set $I_{2k}$ consisting of all positive integers $n$ with $F_{2k} \in Z(n)\cap CG(n)$. Pick $n\in \mathbb{N}$ whose Zeckendorf decomposition is 
$$n\ =\ \underbrace{\sum_{i=2}^{2k-2}c_iF_{i}}_{=: L}+\underbrace{\sum_{i=2k}^\infty c_iF_{i}}_{=: R}\mbox{ }\lowersub{\alpha Z}.$$
By Corollary \ref{toCG}, 
$\max CG(L) \le F_{2k-2}$ and $\min CG(R) \ge F_{2k}$. Let 
$$L\ =\ \sum_{i=1}^{k-1} d_i F_{2i}\mbox{ }\lowersub{\alpha CG}\quad \mbox{ and }\quad R\ =\ \sum_{i=k}^{\infty} d_i F_{2i}\mbox{ }\lowersub{\alpha CG}.$$

We claim that 
$$n \ =\  \underbrace{\sum_{i=1}^{k-1} d_i F_{2i}}_{L} + \underbrace{\sum_{i=k}^\infty d_i F_{2i}}_R$$
is the Chung-Graham decomposition of $n$. Suppose not; then there exist $j_1\le k-1$ and $j_2\ge k$ such that
$$d_{j_1}\ = \ d_{j_2}\ =\ 2\quad \mbox{ and }\quad d_{j_1+1} \ =\ d_{j_1 + 2} \ =\ \cdots \ =\ d_{j_2 - 1} \ =\ 1.$$ 
However, this implies
$$L\ \ge\ 2F_{2j_1} + \sum_{i = j_1 + 1}^{k-1} F_{2i}\ =\ F_{2k-1} + F_{2j_1} - F_{2j_1 -1}\ \ge\ F_{2k-1},$$
which contradicts $L = \sum_{i=2}^{2k-2} c_i F_{i}\mbox{ }\lowersub{\alpha Z}$. 

Therefore, $F_{2k}\in Z(n)\cap CG(n)$ if and only if $F_{2k}\in Z(R)\cap CG(R)$, which is the same as $R\in C_{2k}$. In other words, $n\in I_{2k}$ if and only if $R\in C_{2k}$. This gives
$$I_{2k} \ =\ \{nF_{2k}+\lfloor(n-1)\phi\rfloor F_{2k+1} + j\ :\ n\in \mathbb{N},\, 0\leq j\leq F_{2k-1}-1\},$$
as claimed. 
\end{proof}

\appendix

\section{Proofs of conversion between the two decompositions}\label{apencon}
\begin{proof}[Proof of Lemma \ref{l1}]
We proceed by induction. Let $u:=F_{2k} + \sum_{m=1}^N a_mF_{2k+2m+1}$. For $N = 1$, we have 
$u = F_{2k} + F_{2k+3}\ =\ 2F_{2k+2}$. Suppose that \eqref{e1} is true for all $N\le j$ for some $j\ge 1$. We show that \eqref{e1} is true for $N = j+1$. 

Case 1: $a_m = 0$ for each $m\in [1,j]$. We have
$$u \ =\ F_{2k}+ F_{2k+2j+3} \ =\ 2F_{2k+2} + \sum_{m=2}^{j+1} F_{2k+2m}\mbox{ }\lowersub{\alpha CG}.$$

Case 2: $a_m \neq 0$ for some $m\in [1,j]$. Let $j^* = \max\{1\le m\le j: a_m = 1\}$.
By the inductive hypothesis, we have
\begin{align}\label{e2}
u\ =\ F_{2k} + \sum_{m=1}^{j+1}a_m F_{2k+2m+1}&\ =\ F_{2k} + \sum_{m=1}^{j^*} a_m F_{2k+2m+1} + F_{2k+2j+3}\nonumber\\
&\ =\ bF_{2k+2} + \sum_{m=2}^{j^*} b_mF_{2k+2m}\mbox{ }\lowersub{\alpha CG} + F_{2k+2j+3}.
\end{align}

\begin{enumerate}
    \item[] Case 2.1: $j^* = 1$. Then 
    \begin{align*}
    u \ =\ F_{2k} + F_{2k+3} + F_{2k+2j+3}&\ =\ F_{2k+2} + F_{2k+2} + F_{2k+2j+3}\\
    &\ =\ F_{2k+2} + 2F_{2k+4} + \sum_{m=3}^{j+1}F_{2k+2m}\mbox{ }\lowersub{\alpha CG}.
    \end{align*}
    \item[] Case 2.2: $j^*\ge 2$ and $b_{j^*} = 1$. From \eqref{e2}, we write $u$ as
    \begin{align*}
        &bF_{2k+2} + \sum_{m=2}^{j^*-1} b_mF_{2k+2m} + F_{2k+2j^*} + F_{2k+2j+3}\\
        \ =\ &bF_{2k+2} + \sum_{m=2}^{j^*-1} b_m F_{2k+2m} + 2F_{2k+2j^*+2} + \sum_{m = j^*+2}^{j+1} F_{2k+2m}\mbox{ }\lowersub{\alpha CG}.
    \end{align*}
    \item[] Case 2.3: $j^*\ge 2$ and $b_{j^*} = 2$. From \eqref{e2}, we write $u$ as
    \begin{align*}
        &bF_{2k+2} + \sum_{m=2}^{j^*-1} b_mF_{2k+2m} + 2F_{2k+2j^*} + F_{2k+2j+3}\\
        \ =\ &bF_{2k+2} + \sum_{m=2}^{j^*-1} b_mF_{2k+2m} + F_{2k+2j^*} + 2F_{2k+2j^*+2} +\sum_{m=j^*+2}^{j+1} F_{2k+2m}\mbox{ }\lowersub{\alpha CG}.
    \end{align*}
\end{enumerate}
We have verified \eqref{e1} for all cases. 
\end{proof}

\begin{proof}[Proof of Corollary \ref{toCG}]
First, we prove \eqref{e3}.
    If $2r\le N\le 2r+2$, then 
    $$
    F_{2k} + F_{2k+2r+1}\ =\ 2F_{2k+2} +  \sum_{m=2}^{r} F_{2k+2m},
    $$
    and \eqref{e3} is true.

    Suppose that $N\ge 2r+3$. Let $m_1 < m_2 < \cdots < m_s$ be all even integers in $[2r+3, N]$ with $a_{m_i} = 1$ for each $i\le s$. 
    Let $b_{i,j}$ denote the coefficients of the corresponding Chung-Graham decompositions. By Lemma \ref{l1}, we have
    \begin{align*}
        &F_{2k} + F_{2k+2r+1} + \sum_{m=2r+3}^N a_m F_{2k+m}\\
        \ =\ &F_{2k} + F_{2k+2r+1} + \sum_{m=2r+3}^{m_1-3} a_mF_{2k+m} + \sum_{j=1}^{s-1} \left(F_{2k+m_j} + \sum_{m=m_j+3}^{m_{j+1}-3} a_mF_{2k+m}\right)  \\
        &\ + F_{2k+m_s} + \sum_{m=m_s+3}^{N} a_m F_{2k+m}\\
        \ =\ & b_0F_{2k+2} + \sum_{m=2}^{m_1/2-2} b_{0, m} F_{2k+2m}\mbox{ }\lowersub{\alpha CG}  + \sum_{j=1}^{s-1} \sum_{m = m_j/2}^{m_{j+1}/2-2} b_{j,m} F_{2k+2m}\mbox{ }\lowersub{\alpha CG}\\
        &+ \sum_{m = m_s/2}^{\lfloor N/2\rfloor} b_{s,m}F_{2k+2m}\mbox{ }\lowersub{\alpha CG}\\
        \ =\ & bF_{2k+2} + \sum_{m=4}^{2\lfloor N/2\rfloor} b_m F_{2k+m}\mbox{ }\lowersub{\alpha CG}.
    \end{align*} 

Next, \eqref{e4} is the same as
\begin{equation}\label{e5}F_{2k+2r} + \sum_{m=2r+2}^N a_m F_{2k+m}\mbox{ }\lowersub{\alpha Z}\ =\  \sum_{m=2}^{2\lfloor N/2\rfloor} b_m F_{2k+m}\mbox{ }\lowersub{\alpha CG}.\end{equation}

If $2r\le N\le 2r+1$, then $2\lfloor N/2\rfloor = 2r$, and 
$$F_{2k+2r} + \sum_{m=2r+2}^N a_m F_{2k+m}\mbox{ }\lowersub{\alpha Z}\ =\ F_{2k+2r},$$
so \eqref{e5} is true. Suppose that $N\ge 2r + 2$. Let $m_1 < m_2 < \cdots < m_s$ be all even integers in $[2r+2, N]$ with $a_{m_i} = 1$ for each $i\le s$. 
Let $b_{i,j}$ denote the coefficients of the corresponding Chung-Graham decompositions.
By Lemma \ref{l1}, we have
\begin{align*}
    &F_{2k+2r} + \sum_{m=2r+2}^N a_m F_{2k+m}\\
    \ =\ &F_{2k+2r} + \sum_{m=2r+2}^{m_1-3} a_mF_{2k+m} + \sum_{j=1}^{s-1}\left(F_{2k+m_j} + \sum_{m=m_j+3}^{m_{j+1}-3} a_mF_{2k+m}\right) + \\
    &F_{2k+m_s} + \sum_{m= m_s + 3}^{N} a_m F_{2k+m}\\
    \ =\ &\sum_{m=r}^{m_1/2-2}b_{0,m} F_{2k+2m}\mbox{ }\lowersub{\alpha CG} + \sum_{j=1}^{s-1}\sum_{m = m_j/2}^{m_{j+1}/2-2} b_{j,m}F_{2k+2m}\mbox{ }\lowersub{\alpha CG} + \sum_{m = m_s/2}^{\lfloor N/2\rfloor} b_{s,m}F_{2k+2m}\mbox{ }\lowersub{\alpha CG}\\
    \ =\ &\sum_{m=2}^{2\lfloor N/2\rfloor} b_m F_{2k+m}\mbox{ }\lowersub{\alpha CG}.
\end{align*} 
\end{proof} 

\begin{proof}[Proof of Lemma \ref{toZeck0}, Equation \eqref{e6}]
     We proceed by induction. It is readily verified that \eqref{e6} holds for $M\in \{0, 2, 4\}$. Suppose that \eqref{e6} holds for $M = j$ for some even $j\ge 4$. We show that \eqref{e6} holds for $M = j+2$. We have
    \begin{align*}
        2F_{2k} + \sum_{m=2}^{j+2} b_m F_{2k+m}\mbox{ }\lowersub{\alpha CG} &\ =\ 2F_{2k} + \sum_{m=2}^j b_m F_{2k+m} + b_{j+2} F_{2k+j+2}\\
        & \ =\ F_{2k-2} + \sum_{m=1}^{j+1} a_m F_{2k+m}\mbox{ }\lowersub{\alpha Z} + b_{j+2} F_{2k+j+2}.
    \end{align*}

    Case 1: $b_{j+2} = 1$ and $a_{j+1} = 0$. Then 
    $$2F_{2k} + \sum_{m=2}^{j+2} b_m F_{2k+m}\ =\ F_{2k-2} + \sum_{m=1}^{j} a_m F_{2k+m} +  F_{2k+j+2}\mbox{ }\lowersub{\alpha Z}.$$
    
    Case 2: $b_{j+2} = 1$ and $a_{j+1} = 1$. Then 
    $$2F_{2k} + \sum_{m=2}^{j+2} b_m F_{2k+m}\ =\ F_{2k-2} + \sum_{m=1}^{j} a_m F_{2k+m} +  F_{2k+j+3}\mbox{ }\lowersub{\alpha Z}.$$

    Case 3: $b_{j+2} = 2$. Let $j^* = \max\{m: 2\le m\le j, 2|m, b_{m} = 0\}$. Then
      \begin{align*}
      &2F_{2k} + \sum_{m=2}^{j+2} b_m F_{2k+m}\\
      \ =\ &2F_{2k} + \sum_{m=2}^{j^*-2} b_m F_{2k+m} + \sum_{\substack{j^*+2\le m\le j\\ 2|m}} F_{2k + m} + 2F_{2k+j+2}\\
      \ =\ &F_{2k-2} + \sum_{m=1}^{j^*-1} a_m F_{2k+m}\mbox{ }\lowersub{\alpha Z} + 2F_{2k+j^*+2} + \sum_{\substack{j^*+5\le m\le j+3\\ 2\nmid m}} F_{2k+m}.
      \end{align*}
      \begin{enumerate}
          \item[] Case 3.1: $a_{j^*-1} = 0$. Then
          \begin{align*}&2F_{2k} + \sum_{m=2}^{j+2} b_m F_{2k+m}\\
          \ =\ &F_{2k-2} + \sum_{m=1}^{j^*-2} a_m F_{2k+m} + F_{2k+j^*} + F_{2k+j^*+3} + \sum_{\substack{j^*+5\le m\le j+3\\ 2\nmid m}} F_{2k+m}\mbox{ }\lowersub{\alpha Z}.
          \end{align*}
          \item[] Case 3.2: $a_{j^*-1} = 1$. Then 
          \begin{align*}&2F_{2k} + \sum_{m=2}^{j+2} b_m F_{2k+m}\\
          \ =\ &F_{2k-2} + \sum_{m=1}^{j^*-3} a_m F_{2k+m}  + F_{2k+j^*+1} + F_{2k+j^*+3} + \sum_{\substack{j^*+5\le m\le j+3\\ 2\nmid m}} F_{2k+m}\mbox{ }\lowersub{\alpha Z}.
          \end{align*}
      \end{enumerate}
\end{proof}

\begin{proof}[Proof of Lemma \ref{toZeck0}, Equation \eqref{e7}]
    We prove by induction. It is readily verified that \eqref{e7} is true when $M \in \{0, 2, 4\}$. Assume that \eqref{e7} holds for $M = j$ for some even $j\ge 4$. We show that \eqref{e7} holds for $M = j+2$. 

Case 1: $b_{j+2} = 1$. By the induction hypothesis, we have
\begin{align*}
F_{2k} + \sum_{m=1}^{j+2} b_m F_{2k+m}\mbox{ }\lowersub{\alpha CG}&\ =\ F_{2k} + \sum_{m=1}^{j} b_m F_{2k+m}\mbox{ }\lowersub{\alpha CG} + F_{2k+j+2}\\
&\ =\ \begin{cases}F_{2k-2}\\ F_{2k}\end{cases} + \sum_{m=1}^{j+1} a_m F_{2k+m}\mbox{ }\lowersub{\alpha Z} + F_{2k+j+2}\\
&\ =\ \begin{cases}F_{2k-2}\\ F_{2k}\end{cases} + \sum_{m=1}^{j} a_m F_{2k+m} + \begin{cases}F_{2k+j+2}\mbox{ }\lowersub{\alpha Z}\\ F_{2k+j+3}\mbox{ }\lowersub{\alpha Z}\end{cases}\\
&\ =\ \begin{cases}F_{2k-2}\\ F_{2k}\end{cases} + \sum_{m=1}^{j+3} a_m F_{2k+m} \mbox{ }\lowersub{\alpha Z}.
\end{align*}

Case 2: $b_{j+2} = 2$. Let $A = \{2\le m\le j: 2|m, b_m = 0\}$.

\begin{enumerate}
    \item[] Case 2.1: $A = \emptyset$. Then 
    \begin{align*}
    &F_{2k} + \sum_{m=1}^{j+2} b_m F_{2k+m}\mbox{ }\lowersub{\alpha CG}\\
    \ =\ &F_{2k} + F_{2k+2} + F_{2k+4} + \cdots + F_{2k+j} + 2F_{2k+j+2}\\
    \ =\ &F_{2k-2} + F_{2k+1} + F_{2k+3}  + \cdots + F_{2k+j+1} + F_{2k+j+3}\mbox{ }\lowersub{\alpha Z}.
    \end{align*}
    \item[] Case 2.2: $A\neq \emptyset$. Let $j^* = \max A$. Then by the inductive hypothesis, 
    \begin{align}\label{re21}
        &F_{2k} + \sum_{m=1}^{j+2} b_m F_{2k+m}\mbox{ }\lowersub{\alpha CG}\nonumber\\
        \ =\ &F_{2k} + \sum_{m=2}^{j^*-2} b_m F_{2k+m} + \sum_{\substack{j^*+2\le m\le j\\ 2|m}}  F_{2k+m} + 2F_{2k+j+2}\nonumber\\
        \ =\ &\begin{cases}F_{2k-2}\\ F_{2k}\end{cases} + \sum_{m=1}^{j^*-1} a_m F_{2k+m}\mbox{ }\lowersub{\alpha Z} + F_{2k+j^*} + F_{2k+j^*+3} + \sum_{\substack{j^*+5\le m\le j+3\\ 2\nmid m}} F_{2k+m}.
    \end{align}
    If $a_{j^*-1} = 0$, then \eqref{re21} is the Zeckendorf decomposition of $F_{2k} + \sum_{m=1}^{j+2} b_m F_{2k+m}$.
    If $a_{j^*-1} = 1$, then 
    \begin{align*}
        &F_{2k} + \sum_{m=1}^{j+2} b_m F_{2k+m}\mbox{ }\lowersub{\alpha CG}\\
        \ =\ & \begin{cases}F_{2k-2}\\ F_{2k}\end{cases} + \sum_{m=1}^{j^*-3} a_m F_{2k+m} + F_{2k+j^*+1} + F_{2k+j^*+3} + \sum_{\substack{j^*+5\le m\le j+3\\ 2\nmid m}} F_{2k+m}\mbox{ }\lowersub{\alpha Z}.
    \end{align*}
\end{enumerate}
\end{proof}

\section*{Acknowledgement} The authors thank the anonymous referees for their careful reading and helpful suggestions, which have improved the exposition of the paper.


\ \\
\end{document}